\documentclass[draft]{amsart}
\usepackage{amssymb} 
\newtheorem{theorem}{Theorem}[section]
\newtheorem{lemma}[theorem]{Lemma}
\newtheorem{proposition}[theorem]{Proposition}
\newtheorem{corollary}[theorem]{Corollary}

\theoremstyle{definition}

\theoremstyle{approach}

\numberwithin{equation}{section}

\begin{document}
\setcounter{page}{1}
\title[Bochner-Schoenberg-Eberlein-type inequality
of the direct sum, ideals and quotient of Fr\'echet algebras]{Bochner-Schoenberg-Eberlein-type inequality
 of the direct sum, ideals and quotient of Fr\'echet algebras}
\author[M. Amiri and  A. Rejali]{M. Amiri and A. Rejali}

\begin{abstract}

Let $\mathcal A$ and $\mathcal B$ be two commutative semisimple 
Fr\'echet algebras. We first give a characterization of the multiplier
algebra of the direct sum of $\mathcal A$ and $\mathcal B$. We then
prove that $\mathcal A\oplus \mathcal B$ is a {\text BSE}-algebra 
if and only if $\mathcal A$ and $\mathcal B$ are {\text BSE}-algebras.
Furthermore, for a closed ideal $I$ of $\mathcal A$, we study 
multipliers of ideals and quotient algebras of $\mathcal A$ and show that $I$
and $\frac{\mathcal A}{I}$ are {\text BSE}-algebras, under certain conditions.
\end{abstract}
 
\maketitle \setcounter{section}{0}

\section{Introduction and Preliminaries}

{\text BSE} stands for the theorem Bochner-Schoenberg-Eberlein.  
This theorem proved for locally compact Abelian groups \cite{3} and was 
generalized for commutative Banach algebras, by Takahasi and Hatori \cite{9, 11}. Moreover,  
it developed by other authors such as Inoue and Takahasi \cite{5} and later by 
Kaniuth and \"Ulger \cite{7}. We studied the {\text BSE} 
property for the commutative Fr\'echet algebras; see \cite{1}. 

Let $\mathcal A$ be a commutative Banach algebra without order and the
space $\Delta(\mathcal A)$ denotes the set of all 
nonzero multiplicative linear functionals on $\mathcal A$ 
with respect to the Gelfand-topology. A bounded continuous function $\sigma$ 
on $\Delta(\mathcal A)$ is
called a {\text BSE}-function if there exists a positive real number $\beta$ such 
that for every finite number of complex-numbers $c_1,\cdots, c_n$ and the same 
number of $\varphi_1,\cdots, \varphi_n$ in $\Delta(\mathcal A)$ the inequality
$$\big{\vert}\sum_{i=1}^{n} c_{i}\sigma(\varphi_{i})\big{\vert}\leq\beta
\parallel\sum_{i=1}^{n} c_{i}\varphi_{i}\parallel_{\mathcal A^{\ast}}$$
holds. Following \cite{12}, the {\text BSE} norm of $\sigma$, denoted by $\parallel \sigma\parallel_{\text{BSE}}$, is defined to be the infimum of all such $\beta$. The set of all {\text BSE}-functions is denoted by 
$C_{\text{BSE}}(\Delta(\mathcal A))$ where it is a commutative semisimple Banach algebra, under $\parallel .\parallel_{\text{BSE}}$.

Let us recall from \cite{8,9} that a Fr\'echet space is a completely metrizable locally convex space
where its topology is generated by a translation
invariant metric.
Furthermore, following \cite{4},
a complete topological algebra $\mathcal A$ is a Fr\'echet algebra
if its topology is produced by a countable family of increasing
submultiplicative seminorms $(p_\ell)_{\ell \in \mathbb N}$.
The class of Fr\'echet algebras is 
an important class of locally convex algebras has been widely studied
by many authors. 
Note that every Fr\'echet (algebra) space is not necessarily
a Banach (algebra) space. Some differences between
Banach and Fr\'echet (algebras) spaces introduced in the survey paper \cite{2}.

Let $(\mathcal{A},p_{\ell})_{\ell \in \mathbb N}$ be a Fr\'echet algebra.
Consider $\mathcal{A}^*$ 
the topological dual of $\mathcal A$. The strong topology on $\mathcal{A}^*$
is generated by seminorms $(P_M)$ where $M$ is a bounded set in $\mathcal A$; see \cite{9} for more details.
Following \cite{1}, a bounded complex-valued
continuous function $\sigma$ defined on $\Delta(\mathcal A)$ is called a
{\text BSE}-function, if there exist a bounded set $M$ in $\mathcal A$ and 
a positive real number $\beta_{M}$ such that 
for every finite number of complex-numbers
$c_1,\cdots, c_n$ and the same number of 
$\varphi_1,\cdots, \varphi_n$ in $\Delta(\mathcal A)$ the inequality
$$
\vert\sum_{i=1}^{n} c_{i}\sigma(\varphi_{i})\vert\leq\beta_{M}
P_{M}(\sum_{i=1}^{n}c_{i}\varphi_{i})
$$
holds. Moreover by \cite[Theorem 3.3]{1}, 
$C_{\text{BSE}}(\Delta(\mathcal A))$ is a commutative 
semisimple Fr\'echet algebra and
$$C_{\text{BSE}}(\Delta(\mathcal A))=\mathcal A^{**}
\vert_{\Delta(\mathcal A)}\cap C_{b}(\Delta(\mathcal A)).$$

Let now $(\mathcal A,p_\ell)$ be a commutative Fr\'echet algebra. 
A linear operator $T$ on $\mathcal A$ is called a multiplier if it satisfies 
$aT(b)=T(ab)$, for all $a,b\in \mathcal A$. The set $\mathrm{M}(\mathcal A)$ of 
all multipliers of $\mathcal A$ with the strong operator topology, is a 
commutative unital complete locally convex algebra and not necessarily Fr\'echet algebra; see
\cite[Proposition 4.1]{1}.
Analogous to the Banach 
case, for each $T\in \mathrm{M}(\mathcal A)$, there exists a unique continuous function
$\widehat{T}$ on $\Delta(\mathcal A)$ such that
$\varphi(T(a))=\widehat{T}(\varphi)\varphi(a)$ for all $a\in\mathcal A$ 
and $\varphi \in\Delta(\mathcal A)$. The algebra
$\mathcal A$ is called {\text BSE} algebra
if $$\widehat{\mathrm{M}(\mathcal A)}=C_{\text{BSE}}(\Delta(\mathcal A)),$$ where 
$\widehat{\mathrm{M}(\mathcal A)}=\lbrace\widehat{T}:T\in M(\mathcal A)\rbrace.$
A bounded net $(e_\alpha)_{\alpha\in\Lambda}$ in $\mathcal A$ 
is called a $\Delta$-weak approximate identity for $\mathcal A$
if it satisfies $\varphi(e_{\alpha})\longrightarrow_\alpha 1$ or equivalently 
$\varphi(e_{\alpha}a)\longrightarrow_\alpha \varphi(a)$ for every $a \in \mathcal A$
and $\varphi\in\Delta(\mathcal A)$. Moreover, $\mathcal A$ has a bounded $\Delta$-weak approximate identity if and only 
if $\widehat{\mathrm{M}(\mathcal A)}\subseteq C_{\text{BSE}}(\Delta(\mathcal A)).$
In addition,
$$\mathcal{M}(\mathcal A)=\big{\lbrace} \Phi:\Delta(\mathcal A)
\rightarrow \mathbb{C}: \text{$\Phi$ is continuous and $\Phi\cdot \widehat{\mathcal A} \subseteq \widehat{\mathcal A}$}\big{\rbrace}.$$
This is another definition of the multiplier algebra of Fr\'echet algebras. 
If $\mathcal A$ is a commutative semisimple Fr\'echet algebra, then
$\mathcal{M}(\mathcal A)=\widehat{\mathrm{M}(\mathcal A)}
=\mathrm{M}(\mathcal A)$. 
In the present paper, we show that
$\mathcal{M}(\mathcal A \oplus \mathcal B)=\mathcal{M}(\mathcal A)
\times\mathcal{M}(\mathcal B)$ and
$$C_{\text{BSE}}(\Delta(\mathcal A \oplus\mathcal B))=
C_{\text{BSE}}(\Delta(\mathcal A))\times C_{\text{BSE}}(\Delta(\mathcal B)).$$
Additionally, we prove that 
$\mathcal A\oplus \mathcal B$ is a {\text BSE}-algebra if and
only if $\mathcal A$ and $\mathcal B$ are {\text BSE}. We also
offer conditions under which closed ideals and quotient algebras 
of {\text BSE} algebras are {\text BSE}-algerbras.
\section{The multiplier algebra for the direct sum of Fr\'echet algebras}

The direct sum of Banach algebras studied in \cite{6}. In the following, 
we generalized it for Fr\'echet algebras. Let $(\mathcal A,r_\ell)$ 
and $(\mathcal B,s_\ell)$ be two commutative
Fr\'echet algebras. Then,
$\mathcal A \oplus \mathcal B=\mathcal A \times \mathcal B$ defined by
\begin{enumerate}
	\item[(i)]$(a_1,a_2)\cdot (b_1,b_2)=(a_1 b_1,a_2 b_2), \,\,for\,\,
	a_1,a_2 \in {\mathcal A},\, b_1,b_2 \in {\mathcal B};$
	\item[(ii)]$P_{\ell}(a,b)=r_{\ell} (a)+s_{\ell} (b), \,\,for\,\, a\in{\mathcal
		A},\,b \in{\mathcal B}\,\,and\,\,\ell\in \mathbb N;$
	\item[(iii)]$(a_1,a_2)+(b_1,b_2)=(a_1+ b_1 , a_2+ b_2).$
\end{enumerate}
By a standard argument, the following is straightforward.  

\begin{lemma} 
	Let $(\mathcal A, r_\ell)$ and $(\mathcal B,s_\ell)$ be two Fr\'echet algebras. Then,
	\begin{enumerate}
		\item[(i)]$(\mathcal A\oplus\mathcal B,P_\ell)$ is a Fr\'echet algebra.
		\item[(ii)]$(\mathcal A\oplus\mathcal B)^{\ast}=
		{\mathcal A}^{\ast}\oplus{\mathcal B}^{\ast}$ as homeomorphism.
		\item[(iii)]$\Delta(\mathcal A\oplus \mathcal B)=\big{(}\Delta(\mathcal A)
		\times\lbrace 0 \rbrace\big{)} \bigcup \big{(}\lbrace 0 \rbrace \times \Delta(\mathcal B)\big{)}.$
		\item[(iv)]$\mathcal A\oplus \mathcal B$ is semisimple if and only if
		both $\mathcal A$ and $\mathcal B$ are semisimple.
	\end{enumerate}
\end{lemma}

\begin{lemma}\label{g1}
	Let $\mathcal A$ and $\mathcal B$ be two commutative Fr\'echet algebras. Then,
	$$C_{\text{BSE}}(\Delta(\mathcal A\oplus \mathcal B))=C_{\text{BSE}}
	(\Delta(\mathcal A))\times C_{\text{BSE}}(\Delta(\mathcal B)).$$
\end{lemma}

\begin{proof} 
	Let $\sigma \in C_{\text{BSE}}(\Delta(\mathcal A\oplus \mathcal B))$. Then, by \cite[Proposition 3.5(iii)]{1}
	we have 
	$$\sigma\in C_{b}(\Delta(\mathcal A\oplus \mathcal B))\cap
	(\mathcal A^{**}\oplus \mathcal B^{**})\big{\vert}_{\Delta(\mathcal A\oplus \mathcal B)}.$$ 
	Therefore, there exist $\sigma_{1} \in \mathcal A^{**}$ and $\sigma_{2} \in \mathcal B^{**}$ where  
	$\sigma_{1}\vert_{\Delta(\mathcal A)}\in\mathcal A^{**}\vert_{\Delta(\mathcal A)}$, 
	$\sigma_{2}\vert_{\Delta(\mathcal B)}\in\mathcal B^{**}\vert_{\Delta(\mathcal B)}$ 
	and $\sigma=(\sigma_{1},\sigma_{2})$. In addition, there 
	exist a bounded set $M$ in $\mathcal A \oplus \mathcal B$ and a positive real number $\beta_{M}$
	such 
	that for every finite number of $c_1,\cdots,c_n \in \mathbb{C}$ and 
	$(\varphi_1,\psi_1),\cdots,(\varphi_n,\psi_n)\in 
	\Delta(\mathcal A\oplus \mathcal B)$ the inequality
	$$\big{\vert} \sum_{i=1}^{n}c_i \sigma(\varphi_i,\psi_i)\big{\vert}
	\leq \beta_{M} P_{M}\big{(}\sum_{i=1}^{n}c_i(\varphi_i, \psi_i)\big{)}$$
	holds. In particular, for each $(\varphi_1,0),\cdots,(\varphi_n,0)\in 
	\Delta(\mathcal A\oplus \mathcal B)$ and $c_1,\cdots,c_n \in \mathbb{C}$, there
	exist bounded sets $N_1\subseteq \mathcal A$ and $N_2\subseteq \mathcal B$
	such that
	\begin{align*}
	\big{\vert} \sum_{i=1}^{n}c_i \sigma_{1}(\varphi_i)\big{\vert}&=\big{\vert} \sum_{i=1}^{n}c_i \sigma(\varphi_i,0)\big{\vert} \\
	&\leq \beta_{M} P_{M}\big{(}\sum_{i=1}^{n}c_i(\varphi_i,0)\big{)} \\
	&={\beta}_{M}\sup\big{\lbrace}\big{\vert} \sum_{i=1}^{n}c_i(\varphi_i,0)(a,b)\big{\vert}: (a,b)\in M \big{\rbrace} \\
	&={\beta}_{M}\sup\big{\lbrace}\big{\vert} \sum_{i=1}^{n}c_i\varphi_i(a)\big{\vert}: a\in N_{1} \big{\rbrace} \\
	&=\beta_{M} P_{N_{1}}(\sum_{i=1}^{n}c_i\varphi_i), 
	\end{align*}
	and similarly
	$$
	\big{\vert}\sum_{i=1}^{n}c_i \sigma_{2}(\psi_i)\big{\vert}=\big{\vert} \sum_{i=1}^{n}c_i \sigma(0,\psi_{i})\big{\vert} 
	\leq\beta_{M} P_{N_{2}}(\sum_{i=1}^{n}c_i\psi_{i}).
	$$
	We recall from \cite{9} that
	$N_1=\pi_1(M)$ and $N_2= \pi_2 (M)$
	where $\pi_1$ and $\pi_1$ are projection mappings on $\mathcal A\oplus \mathcal B$.
	Moreover, $M\subseteq N_1\times N_2$.
	By above arguments, $\sigma_1 \in C_{\text{BSE}}(\Delta
	(\mathcal A))$ and $\sigma_2 \in C_{\text{BSE}}(\Delta(\mathcal B))$, 
	where $\sigma_1(\varphi)=\sigma(\varphi,0)$ and
	$\sigma_2 (\psi)=\sigma(0,\psi)$ for each $\varphi \in \Delta(\mathcal A)$
	and $\psi \in \Delta(\mathcal B)$. This implies that 
	$$C_{\text{BSE}}(\Delta(\mathcal A\oplus \mathcal B))\subseteq C_{\text{BSE}}
	(\Delta(\mathcal A))\times C_{\text{BSE}}(\Delta(\mathcal B)).$$
	For the reverse conclusion, let $\sigma_1 \in C_{\text{BSE}}(\Delta(\mathcal A))$
	and $\sigma_2 \in C_{\text{BSE}}(\Delta(\mathcal B))$. If $\xi\in \Delta(\mathcal A
	\oplus \mathcal B)$, then $\xi=(\varphi ,0)$ or $\xi=(0,\psi)$ 
	for some $\varphi \in \Delta(\mathcal A)$ or $\psi \in \Delta
	(\mathcal B)$. We define
	\begin{center}
		$\sigma(\xi):=\left\{
		\begin{array}{rl}
		\sigma_1(\varphi) & \text{if $\;\xi=(\varphi,0)$},\\
		\sigma_2(\psi) & \text{if $\;\xi=(0,\psi)$}.
		\end{array}\right.$
	\end{center}
	By assumption, there exist a bounded set $N_{1}$ of $\mathcal A$ and 
	a positive real number $\beta_{N_{1}}$ such that for every finite number
	of complex-numbers $c_1,\cdots, c_n$ and the same number of 
	$\varphi_1,\cdots, \varphi_n$ in $\Delta(\mathcal A)$ the inequality
	$$\big{\vert} \sum_{i=1}^{n}c_i \sigma_{1}(\varphi_i)\big{\vert} \leq\beta_{N_{1}} P_{N_{1}}(\sum_{i=1}^{n}c_i\varphi_i)$$ 
	holds. Also,
	there exist a bounded set $N_{2}$ of $\mathcal B$ and 
	a positive real number $\beta_{N_{2}}$ such that for every finite number
	of complex-numbers $c_1,\cdots, c_n$ and the same number of $\psi_1,\cdots, \psi_n$ in $\Delta(\mathcal B)$
	the inequality
	$$\big{\vert} \sum_{i=1}^{n}c_i \sigma_{2}(\psi_i)\big{\vert} \leq\beta_{N_{2}} P_{N_{2}}(\sum_{i=1}^{n}c_i\psi_i)$$
	holds. 
	Therefore, $$\big{\vert} \sum_{i=1}^{n}c_i \sigma(\xi)\big{\vert} \leq \beta_{M} P_{M}(\sum_{i=1}^{n}c_i(\xi)),$$ 
	where $M=N_1 \times N_2$ and $\beta_{M}=\max\lbrace\beta_{N_{1}},\beta_{N_{2}}\rbrace$. 
	Consequently, $$\sigma \in C_{\text{BSE}}(\Delta(\mathcal A \oplus \mathcal B)),$$ which completes the proof.
\end{proof}

\begin{proposition}\label{g2}
	Let $\mathcal A$ and $\mathcal B$ be two commutative Fr\'echet algebras. Then,
	$$\mathcal{M}(\mathcal A \oplus \mathcal B)=\mathcal{M}
	(\mathcal A)\times\mathcal{M}(\mathcal B).$$
\end{proposition}

\begin{proof} 
	Let $\Phi \in \mathcal{M}(\mathcal A)$ and $\Psi \in \mathcal{M}
	(\mathcal B)$. Since 
	$\Phi \cdot \widehat{\mathcal A}\subseteq \widehat{\mathcal A}$ and 
	$\Psi \cdot \widehat{\mathcal B} \subseteq \widehat{\mathcal B}$, for 
	all $(a,b)\in\mathcal A \oplus \mathcal B$ there are elements $c\in \mathcal A$ 
	and $d\in \mathcal B$ such that
	\begin{align*}
	\big{(}(\Phi,\Psi)\cdot\widehat{(a,b)}\big{)}(\varphi,0)&=(\Phi,\Psi)(\varphi,0)\widehat{(a,b)}(\varphi,0)\\
	&=\Phi(\varphi)\widehat{a}(\varphi)\\
	&=\widehat{c}(\varphi),
	\end{align*}
	where $(\Phi,\Psi)(\varphi,0)=\Phi(\varphi)$ for each 
	$(\varphi,0)\in \Delta(\mathcal A)
	\times\lbrace 0 \rbrace$. Similarly,
	\begin{align*}
	\big{(}(\Phi,\Psi)\cdot\widehat{(a,b)}\big{)}(0,\psi)&=(\Phi,\Psi)(0,\psi)\widehat{(a,b)}(0,\psi)\\
	&=\Psi(\psi)\widehat{b}(\psi)\\&
	=\widehat{d}(\psi),
	\end{align*}
	where $(\Phi,\Psi)(0,\psi)=\Psi(\psi)$ for each $(0,\psi)\in \lbrace 0 \rbrace \times \Delta(\mathcal B).$ Hence, $$\big{(}(\Phi,\Psi)\cdot\widehat{(a,b)}\big{)}(\varphi,0)=
	\widehat{(c,d)}(\varphi,0),$$ and $$\big{(}(\Phi,\Psi)\cdot\widehat{(a,b)}\big{)}(0,\psi)=
	\widehat{(c,d)}(0,\psi).$$ Thus, $(\Phi,\Psi)\cdot\widehat{(\mathcal A\oplus \mathcal B)}\subseteq 
	\widehat{(\mathcal A\oplus \mathcal B)}$ and $(\Phi,\Psi)\in \mathcal{M}(\mathcal A\oplus \mathcal B)$. 
	
	Now, suppose that 
	$F \in \mathcal{M}(\mathcal A\oplus \mathcal B)$. Define 
	$\Phi(\varphi)=F(\varphi,0)$ and $\Psi(\psi)= F(0,\psi)$ for all 
	$\varphi \in\Delta(\mathcal A)$ and $ \psi \in \Delta(\mathcal B).$ Therefore, 
	$F=(\Phi,\Psi)$. It is enough to show that $\Phi \in \mathcal{M}(\mathcal A)$ and $\Psi \in 
	\mathcal{M}(\mathcal B).$ For each $a\in\mathcal A$, there exists
	$(a^{\prime},b^{\prime})\in\mathcal A\oplus\mathcal B$ such that 
	$$\Phi(\varphi)\widehat{a}(\varphi)=F(\varphi, 0)\widehat{(a,0)}(\varphi,0)=\widehat{(a^{\prime}, b^{\prime})}(\varphi,0)
	=\widehat{a^{\prime}}(\varphi).$$
	Consequently, $\Phi \cdot \widehat{\mathcal A}\subseteq \widehat{\mathcal A}$ and $\Phi \in \mathcal{M}(\mathcal A).$
	Similarly, $\Psi \in\mathcal{M}(\mathcal B).$ Hence, 
	$$\mathcal{M}(\mathcal A \oplus \mathcal B)=\big{\lbrace}(\Phi,\Psi):
	\Phi \in \mathcal{M}(\mathcal A),\, \Psi \in \mathcal{M}
	(\mathcal B)\big{\rbrace},$$
	and completes the proof.
\end{proof}

\begin{corollary}
	Let $\mathcal A$ and $\mathcal B$ be two commutative semisimple 
	Fr\'echet algebras. Then,
	$$\mathrm{M}(\mathcal A\oplus \mathcal B)=\mathrm{M}(\mathcal A)\oplus \mathrm{M}(\mathcal B).$$
\end{corollary}

We now state the main result of this paper.

\begin{theorem}
	Let $(\mathcal A,r_\ell)$ and $(\mathcal B,s_\ell)$ be two commutative semisimple 
	Fr\'echet algebras. Then, $\mathcal A\oplus \mathcal B$ is a {\text BSE}-algebra 
	if and only if $\mathcal A$ and $\mathcal B$ are {\text BSE}-algebras.
\end{theorem}

\begin{proof} 
	Let $\mathcal A$ and $\mathcal B$ be {\text BSE}-algebras. By applying \cite[Theorem 4.5]{1},
	$\mathcal A$ and $\mathcal B$ have bounded $\Delta$-weak approximate identities. Suppose that  
	$(e_\alpha)_{\alpha}$ and $(f_\beta)_{\beta}$
	are bounded $\Delta$-weak approximate identities
	of $\mathcal A$ and $\mathcal B$, respectively. Therefore, $\lbrace
	(e_\alpha, f_\beta)\rbrace_{(\alpha,\beta)}$
	is a bounded $\Delta$-weak approximate identity for $\mathcal A
	\oplus \mathcal B$. Indeed, 
	for all $\xi \in \Delta(\mathcal A\oplus \mathcal B)$ there exists
	$\varphi \in \Delta(\mathcal A)$ or $\psi \in \Delta(\mathcal B),$ where 
	$\xi=(\varphi,0)$ or $\xi=(0,\psi).$ Thus,
	$$\underset{(\alpha,\beta)}{\lim}\xi(e_\alpha,f_\beta )
	=\underset{(\alpha,\beta)}{\lim}(\varphi,0)(e_\alpha,f_\beta)
	=\underset{\alpha}{\lim}\varphi(e_\alpha)=1,$$
	or
	$$\underset{(\alpha,\beta)}{\lim}\xi(e_\alpha, f_\beta)
	=\underset{(\alpha,\beta)}{\lim}(0,\psi)(e_\alpha,f_\beta)=
	\underset{\beta}{\lim}\psi(f_\beta)=1.$$
	Moreover, for each $\ell \in \mathbb{N}$ we have
	$\underset{\alpha}{\sup}\, r_{\ell}(e_\alpha)< \infty$ and $\underset{\beta}{\sup}\, s_{\ell}(f_\beta)< \infty$.
	Also, $(\mathcal A\oplus \mathcal B, P_\ell)$ is a Fr\'echet algebra,
	where $P_{\ell}(a,b)=r_{\ell}(a)+s_{\ell}(b)$ for each $\ell\in\mathbb{N}$. Therefore,
	$$\underset{(\alpha,\beta)}{\sup}\, P_{\ell}(e_\alpha, f_\beta)=
	\underset{\alpha}{\sup}\, r_{\ell}(e_\alpha) +\underset{\beta}
	{\sup}\, s_{\ell}(f_\beta)< \infty.$$
	Hence, for all $\xi \in \Delta(\mathcal A\oplus \mathcal B)$ we have 
	$\underset{(\alpha,\beta)}{\lim} \xi (e_\alpha,f_\beta )=1$ 
	and $\lbrace(e_\alpha, f_\beta) \rbrace_
	{(\alpha,\beta)}$ is a $\Delta$-weak approximate identity for 
	$\mathcal A\oplus \mathcal B$. Consequently,  
	$$\mathcal{M}(\mathcal A \oplus \mathcal B) \subseteq C_{\text{BSE}}(\Delta
	(\mathcal A \oplus \mathcal B)).$$ 
	
	For the reverse
	conclusion, let $\sigma \in C_{\text{BSE}}(\Delta(\mathcal A
	\oplus \mathcal B))$. Hence, there exist
	$\sigma_1 \in C_{\text{BSE}}(\Delta(\mathcal A))$ and $\sigma_2 
	\in C_{\text{BSE}}(\Delta(\mathcal B))$ such that 
	$\sigma(\varphi,\psi)=\sigma_{1}(\varphi)+\sigma_{2}(\psi)$ for 
	all $\varphi\in\Delta(\mathcal A)$ and $\psi\in\Delta(\mathcal B).$ 
	Since $\mathcal A$ and $\mathcal B$ are {\text BSE} algebras, 
	$\sigma_1 \in\mathcal{M}(\mathcal A)$ and $\sigma_2 \in\mathcal{M}(\mathcal B)$,
	by applying Proposition \ref{g2}. 
	Therefore, $\sigma \in\mathcal{M}(\mathcal A)\times \mathcal{M}(\mathcal B).$ Thus,  
	$\sigma \in\mathcal{M}(\mathcal A\oplus \mathcal B)$ and 
	$$C_{\text{BSE}}(\Delta(\mathcal A\oplus \mathcal B))\subseteq\mathcal
	{M}(\mathcal A\oplus \mathcal B).$$
	Consequently, $\mathcal A\oplus \mathcal B$ is a {\text BSE}-algebra.
	
	Conversely, suppose that $\mathcal A\oplus \mathcal B$ is a 
	{\text BSE}-algebra. Then, 
	$$\mathcal{M}(\mathcal A\oplus \mathcal B)=\mathrm{M}(\mathcal A\oplus
	\mathcal B)=C_{\text{BSE}}(\Delta(\mathcal A\oplus \mathcal B)).$$
	Let $\sigma_1 \in C_{\text{BSE}}(\Delta(\mathcal A))$ and $\sigma_2 \in
	C_{\text{BSE}}(\Delta(\mathcal B))$. Then $(\sigma_1,0)$ and $(0, \sigma_2)$ belong to
	$C_{\text{BSE}}(\Delta(\mathcal A\oplus \mathcal B)).$ Also, 
	$\sigma_1 \in \mathcal{M}(\mathcal A)$ and $\sigma_2 \in 
	\mathcal{M}(\mathcal B)$ and so  
	$$C_{\text{BSE}}(\Delta(\mathcal A))\subseteq\mathcal{M}(\mathcal A)\;\;\;\;\text{and}\;\;\;\;
	C_{\text{BSE}}(\Delta(\mathcal B))\subseteq\mathcal{M}(\mathcal B).$$
	Conversely, suppose that $\sigma_1 \in \mathcal{M}(\mathcal A)$ and 
	$\sigma_2 \in \mathcal{M}(\mathcal B)$. Thus,
	$(0, \sigma_2)$ and $(\sigma_1 , 0)$ belong to $\mathcal{M}(\mathcal A)\times
	\mathcal{M}(\mathcal B)$, where
	\begin{align*} 
	\mathcal{M}(\mathcal A)\times
	\mathcal{M}(\mathcal B)&=\mathcal{M}(\mathcal A\oplus \mathcal B)\\
	&=C_{\text{BSE}}(\Delta(\mathcal A\oplus \mathcal B))\\
	&=C_{\text{BSE}}(\Delta(\mathcal A))\times C_{\text{BSE}}(\Delta(\mathcal B)).
	\end{align*}
	Thus, $\sigma_1 \in C_{\text{BSE}}(\Delta(\mathcal A))$ and $\sigma_2 
	\in C_{\text{BSE}}(\Delta(\mathcal B))$. Hence, $\mathcal{M}(\mathcal A)\subseteq
	C_{\text{BSE}}(\Delta(\mathcal A))$ and $\mathcal{M}(\mathcal B)\subseteq
	C_{\text{BSE}}(\Delta(\mathcal B))$. Consequently, $\mathcal A$ and $\mathcal B$ are {\text BSE} algebras.
\end{proof}

\section{ Ideals and Quotient algebras of {\text BSE}-Fr\'echet algerbras}

Let $(\mathcal A,p_\ell)$ be a commutative Fr\'echet algebra and $I$ be
a closed ideal of $\mathcal A$. Then, $(I, p_\ell \vert_{I})$ and
$(\frac{\mathcal A}{I},s_\ell)$ are also  Fr\'echet algebras where
$$s_\ell(x+I):=\inf \lbrace p_\ell(x + y): y \in I  \rbrace.$$
In the sequel we will call $I$ an essential ideal, when $I$ equals the closed
linear span of $\lbrace ax: a \in  \mathcal A,\; x \in I \rbrace$.
Moreover, we will give conditions under which $I$ and 
$\frac{\mathcal A}{I}$ are {\text BSE} algerbras.

Following \cite[Corollary 26.25]{9}, if $\mathcal A$ is a Fr\'echet Schwartz space and $I$ be
a closed ideal of $\mathcal A$, then 
$$I^{\ast}\cong \frac{\mathcal A^{\ast}}{I^{\circ}}\quad and \quad(\frac{\mathcal A}{I} )^{\ast}\cong I^{\circ},$$
where $I^{\circ}=\lbrace f \in \mathcal A^{\ast}:\vert f(x) \vert \leq 1, \;\text{for all $x \in I$} \rbrace$
and equivalently $$I^{\circ}=\lbrace f \in \mathcal A^{\ast}:\, f(x)=0, \;\text{for all $x \in I$} \rbrace.$$

Analogous to the Banach case, if $\mathcal A$ is a {\text BSE}-algerbra with
discrete carrier space, then $C_{b}(\Delta(\mathcal A))=\ell_{\infty}(\Delta(\mathcal A))$.
Hence,
$C_{\text{BSE}}(\Delta(\mathcal A))={\mathcal A}^{\ast \ast} \vert_{\Delta(\mathcal A)}.$
Moreover, by similar arguments to the proof of \cite[Theorem 3.1.18]{10}, if $I$ is a closed ideal 
of $\mathcal A$, then $\Delta(I)$ is discrete.

\begin{lemma} \label{i}
	If the Fr\'echet algebra $\mathcal A$  has a bounded $\Delta$-weak approximate identity and
	I is a closed essential ideal of $\mathcal A$, then $\frac{\mathcal A}{I}$ has a bounded $\Delta$-weak approximate identity.
\end{lemma}

\begin{proof} Let $(e_\alpha)_\alpha$ be a bounded $\Delta$-weak 
	approximate identity of $\mathcal A$. Then, for each $\varphi \in \Delta(\mathcal A)$, $\varphi(e_\alpha)\longrightarrow_\alpha 1$. For each $\alpha$, we define $f_\alpha:=e_\alpha+I$ and show that the net $(f_\alpha)_\alpha$ is a bounded $\Delta$-weak approximate identity for $\frac{\mathcal A}{I}$. Suppose that $\psi \in \Delta(\frac{\mathcal A}{I})$. Therefore, there exists $\varphi \in \Delta(\mathcal A)$ such that $\varphi \vert_{I}=0$ and $\psi(a+I):=\varphi(a)$, for each $a \in \mathcal A$. Hence,
	$$\psi(f_\alpha)=\varphi(e_\alpha)\longrightarrow_\alpha 1.$$
	Thus, $(f_\alpha)_\alpha$ is a bounded $\Delta$-weak approximate identity for $\frac{\mathcal A}{I}$.
\end{proof}

Let $(\mathcal A,p_\ell)$ be a Fr\'echet algebra. For each $\ell\in\mathbb{N}$, consider
$$M_{\ell}=\{a\in\mathcal{A}:p_{\ell}(a)< 1\},$$
and
$$
M_{\ell}^{\circ}=\{f\in\mathcal{A}^*:|f(a)|\leq 1\;\text{for all $a\in M_{\ell}$}\}.
$$
As mentioned in section \ref{S0}, by applying \cite[Theorem 3.3]{1}, $\big{(}C_{\text{BSE}}(\Delta(\mathcal A)),r_{\ell}\big{)}$
is a commutative semisimple Fr\'echet algebra such that for each $\ell\in\mathbb{N}$ and $\sigma\in\Delta(\mathcal A)$ we have
$$
r_{\ell}(\sigma)=\sup\{|\sigma(f)|: f\in M_{\ell}^{\circ}\cap <\Delta(\mathcal A)>\}.
$$
Now, the following result is immediate.

\begin{theorem}
	Let $\mathcal A$ be a {\text BSE}-Fr\'echet Schwartz algebra with discrete carrier
	space and $I$ be an essential closed ideal of $\mathcal A$. Then,
	\begin{enumerate}
		\item[(i)]$ C_{\text{BSE}}(\Delta(I))\subseteq \widehat{\mathrm{M}(I)}.$
		\item[(ii)]$ \widehat{\mathrm{M}(\frac{\mathcal A}{I})}=C_{\text{BSE}}(\Delta(\frac{\mathcal A}{I})).$
	\end{enumerate}
\end{theorem}

\begin{proof}
	(i) Suppose that $w \in C_{\text{BSE}}(\Delta(I))$. Note that
	$$\Delta(I)=\lbrace \varphi \vert_{I}: \varphi \in \Delta(\mathcal A)\setminus I^{\circ} \rbrace.$$
	We define $\sigma \in \Delta(\mathcal A)$ as follow 
	\begin{equation*}
	\sigma(\varphi):=\left\{
	\begin{array}{rl}
	w( \varphi \vert_{I}) & \varphi \in \Delta (\mathcal A)\setminus I^{\circ}, \\
	0\;\;\;\;\;\;\;\; & \varphi \in \Delta(\mathcal A)\cap I^{\circ}. 
	\end{array} \right.
	\end{equation*}
	We show that $\sigma \in C_{\text{BSE}}(\Delta(\mathcal A))$.
	In fact, there exists a bounded set $M$ in $I$ and 
	a positive real number $\beta_{M}$ such that for every finite number of complex-numbers 
	$c_1,\cdots, c_n$ and the same number of $\varphi_{1},\cdots, \varphi_{n}$ in $\Delta(\mathcal A)$,   
	we have 
	\begin{align*} 
	\big{\vert}  \sum_{i=1}^{n} c_i \sigma (\varphi_i) \big{\vert} &=\big{\vert}  \sum_{\varphi_i \in \Delta(\mathcal A)\setminus I^{\circ}}c_i \sigma(\varphi_i) \big{\vert} \\
	&=\big{\vert}  \sum_{\varphi_i \in \Delta(\mathcal A)\setminus I^{\circ}}c_iw(\varphi_i \vert_{I}) \big{\vert}\\
	&\leq \beta_{M}P_{M}\big{(}\sum_{\varphi_i \in \Delta(\mathcal A)\setminus I^{\circ}}c_i \varphi_i \vert_{I}\big{)} \\
	&=\beta_{M}P_{M}(\sum_{i=1}^{n}c_i \varphi_i \vert_{I}) \\
	&\leq \beta_{M}P_{M}(\sum_{i=1}^{n}c_i \varphi_i).
	\end{align*}
	Thus, $\sigma \in C_{\text{BSE}}(\Delta(\mathcal A))$.
	Since $\mathcal A$ is {\text BSE}-algebra, take $T \in \mathrm{M}(\mathcal A)$ 
	such that $\widehat{T}=\sigma$ and put 
	$S=T \vert_{I}$. Then, $S \in \mathrm{M}(I)$. Also, for any $x \in I$ and
	$\varphi \in \Delta(\mathcal A)\setminus I^{\circ},$ we have 
	\begin{align*}
	\widehat{S}(\varphi \vert_{I})\widehat{x}(\varphi \vert_{I})&=(\widehat{Sx})(\varphi \vert_{I}) 
	=(\widehat{Tx})(\varphi) 
	=\widehat{T} ( \varphi ) \widehat{x}(\varphi) \\
	&=\sigma(\varphi)\widehat{x}(\varphi \vert_{I})
	=w(\varphi \vert_{I})\widehat{x}(\varphi \vert_{I}).
	\end{align*}
	Since $\varphi \vert_{I} \in \Delta(\mathcal A)$, there exists $x \in I$ such that $ \varphi \vert_{I} (x) \neq 0$. Then, $\widehat{x}(\varphi \vert_{I})\neq 0$.   Therefore, $\widehat{S}(\varphi \vert_{I})=w(\varphi \vert_{I})$. Consequently,
	$w=\widehat{S}$, where $\widehat{S} \in \widehat{\mathrm{M}(I)}$. Hence,
	$C_{\text{BSE}}(\Delta(I))\subseteq \widehat{\mathrm{M}(I)}$. 
	
	(ii) Since $\mathcal A$ is {\text BSE} algebra, it has a bounded $\Delta$-weak approximate identity 
	and by Lemma \ref{i}, $\frac{\mathcal A}{I}$ also has an approximate identity. Thus, 
	$\widehat{\mathrm{M}(\frac{\mathcal A}{I})} \subseteq C_{\text{BSE}}(\Delta (\frac{\mathcal A}{I}))$. 
	To show the reverse inclusion, let $\sigma^{\prime}\in C_{\text{BSE}}(\Delta(\frac{\mathcal A}{I}))$. 
	Since both $\Delta(\mathcal A)$ and
	$\Delta(\frac{\mathcal A}{I})$ are discrete, 
	$$C_{\text{BSE}}(\Delta(\mathcal A))={\mathcal A}^{\ast \ast} \vert_{\Delta(\mathcal A)},$$
	and
	$$C_{\text{BSE}}(\Delta(\frac{\mathcal A}{I}))=(\frac{\mathcal A}{I})^{\ast \ast} \vert_{\Delta (\frac{\mathcal A}{I})}.$$
	Thus, we can take $F^{\prime}\in(\frac{\mathcal A}{I})^{\ast \ast}$ 
	with $F^{\prime} \vert_{\Delta (\frac{\mathcal A}{I})}=\sigma^{\prime}.$ Set $I^{\circ}=\lbrace f\in \mathcal A^{\ast}:f\vert_{I}=0 \rbrace.$ For each $f \in I^{\circ},$ define $f^{\prime}$ by the 
	relation $f^{\prime}(a^{\prime})=f(a)$ $(a^{\prime}=a+I\in\frac{\mathcal A}{I}).$ The map $f\mapsto f^{\prime}$ is an isometric isomorphism of $I^{\circ}$ onto $(\frac{\mathcal A}{I})^{\ast}.$ We define $\widetilde{F^{\prime}}$ by the relation $\widetilde{F^{\prime}}(f)=F^{\prime}(f^{\prime})$ $(f \in I^{\circ}).$ Then, $\widetilde{F^{\prime}} \in ({I^{\circ}})^{\ast}$ and there exists $F \in{\mathcal A}^{\ast \ast}$ such that $r_{\ell}(F)=r_{\ell}(\widetilde{F^{\prime}})$ and $F\vert_{I^{\circ}}=\widetilde{F^{'}}$ by using
	the general framework of the Hahn-Banach theorem
	\cite[Proposition 22.12]{9}. 
	By the 
	{\text BSE} property of $\mathcal A$, there exists $T\in \mathrm{M}(\mathcal A)$ with 
	$\widehat{T}=F\vert_{\Delta(\mathcal A)}$. Since $I$ is essential 
	and $T(ab)=aTb$ for all $a, b \in \mathcal A$, we have $T(I)\subseteq I$. Then, $T^{\prime}$
	defined by $T^{\prime}(a^{\prime})=(Ta)^{\prime}$ $(a\in \mathcal A)$ belongs to $\mathrm{M}(\frac{\mathcal A}{I})$. 
	In this case, $\widehat{T^{\prime}}=\sigma^{\prime}$. Actually for any $a \in \mathcal A$ and $\varphi \in \Delta(\mathcal A)\bigcap I^{\circ},$ we have
	\begin{align*}
	\widehat{a}(\varphi) \widehat{T^{\prime}}(\varphi^{\prime})&=\widehat{a^{\prime}}(\varphi^{\prime})\widehat
	{T^{\prime}}(\varphi ^{\prime})=
	\widehat{(T^{\prime}a^{\prime})}(\varphi^{\prime})\\
	&=\widehat{(Ta)^{\prime}}(\varphi^{\prime})=\widehat{(Ta)}(\varphi)=\widehat{a}(\varphi)\widehat{T}(\varphi),
	\end{align*} 
	so that $\widehat{T^{\prime}}=\widehat{T}(\varphi)$ and 
	\begin{center}
		$\sigma^{\prime}(\varphi^{\prime})=F^{\prime}(\varphi^{\prime})=\widetilde{F^{\prime}}(\varphi)=F(\varphi)=
		\widehat{T}(\varphi)=\widehat{T^{\prime}}(\varphi ^{\prime}).$
	\end{center}
	Therefore, $\sigma^{\prime}=\widehat{T^{\prime}}$, since ${\Delta(\frac{\mathcal A}{I})}
	=\lbrace \varphi^{\prime}:\varphi \in \Delta(\mathcal A)\bigcap I^{\circ}\rbrace.$ 
	Thus, $\sigma^{\prime} \in \widehat{\mathrm{M}({\frac{\mathcal A}{I}})}$. 
	Hence, $C_{\text{BSE}}(\Delta (\frac{\mathcal A}{I})) \subseteq \widehat{\mathrm{M}(\frac{\mathcal A}{I})}.$ 
	Consequently, $\frac{\mathcal A}{I}$ is {\text BSE}.
\end{proof}

\begin{corollary}
	In above Theorem, if $I$ has a bounded $\Delta$-weak approximate identity, then $I$ is a {\text BSE} algebra.
\end{corollary}

\vspace{9mm}

{\footnotesize \noindent

\noindent
 M. Amiri\\
Department of Pure Mathematics,
   University of Isfahan,
    Isfahan, Iran\\
    mitra.amiri@sci.ui.ac.ir\\ 
    mitra75amiri@gmail.com\\
     
\noindent
 A. Rejali\\
Department of Pure Mathematics,
   University of Isfahan,
    Isfahan, Iran\\
    rejali@sci.ui.ac.ir\\

\end{document}